\documentclass[a4paper]{amsart}
\usepackage[english]{babel}
\usepackage[utf8x]{inputenc}
\usepackage[T1]{fontenc}
\usepackage{amsthm}
\usepackage{dsfont}
\usepackage{amssymb}
\usepackage{amsmath}
\theoremstyle{plain}
\usepackage{chngcntr}
\usepackage{relsize}
\usepackage{pdfpages}

\newtheorem{Lemma}{Lemma}
\newtheorem{theorem}[Lemma]{Theorem}
\newtheorem{Proposition}[Lemma]{Proposition}

\newtheorem{Definition}[Lemma]{Definition}
\newtheorem{remark}{Remark}

\usepackage[a4paper,top=3cm,bottom=2cm,left=3cm,right=3cm,marginparwidth=1.75cm]{geometry}

\usepackage{float}
\usepackage{nccmath}
\usepackage{mathtools}
\usepackage{amsfonts}
\usepackage{amsmath}
\usepackage{graphicx}
\usepackage[colorinlistoftodos]{todonotes}
\usepackage[colorlinks=true, allcolors=blue]{hyperref}
\makeatletter
\newcommand*{\rom}[1]{\expandafter\@slowromancap\romannumeral #1@}
\makeatother
\title{Burgess-type volume dependent bounds for character sums over $\mathbb{F}_{p^n}$}
\author{Aishik Chattopadhyay}
\address{Aishik Chattopadhyay,
Ramakrishna Mission Vivekananda Educational and Research Institute, Department of Mathematics, G. T. Road, PO Belur Math, Howrah, West Bengal 711202, India}
\email{aishik.ch@gmail.com}
\begin{document}
\begin{abstract}
We establish a Burgess-type bound for short multiplicative character sums over finite fields $\mathbb{F}_{p^n}$. Let
\[
B=\left\{\sum_{i=1}^{n}x_i\omega_i:
N_i+1\le x_i\le N_i+H_i,1\le i\le n\right\}\subseteq\mathbb{F}_{p^n},
\]
where $1\le H_i\le p$ for all $1\le i\le n$, and the side lengths satisfy
$H_1\le H_2\le\cdots\le H_n.$
We prove that if the side lengths satisfy certain lower bounds in terms of the two largest side lengths, then a nontrivial cancellation occurs in the character sum over the boxes. This generalizes the work of Gabdullin \cite{GB} in dimensions $n=2,3$ to arbitrary dimension. This also generalizes the character sum estimate of Konyagin \cite{Kon} where each of the side lengths of the boxes are greater than $p^{1/4}$. The proof combines techniques from the geometry of numbers, multiplicative energy estimates, and Katz's bounds for multiplicative character sums.
\end{abstract}

\maketitle
\tableofcontents
\section{Introduction}

Let $p$ be a prime, and let $\mathbb{F}_{p^{n}}$ denote the finite field with $p^{n}$ elements. Let us define intervals of length $H_i$ in $\mathbb{Z}$ as 
\[ I_i=[N_i+1,N_i+H_i]\cap \mathbb{Z}.\]  Fix a basis $\{\omega_{1},\ldots,\omega_{n}\}$ of $\mathbb{F}_{p^{n}}$ over $\mathbb{F}_{p}$. For integers $N_i$ and $H_i$ satisfying
\[
1 \leq H_i \leq p \qquad (1 \leq i \leq n),
\]
we define the box $B$ by
\begin{equation}\label{B}
B=\left\{ \sum_{i=1}^{n} x_i\omega_i :
N_i+1 \leq x_i \leq N_i+H_i,\; 1 \leq i \leq n \right\}
\subseteq \mathbb{F}_{p^{n}},
\end{equation}
and the volume of the box as 
\[|B|=H_1\, H_2 \cdots H_n \]
where the $H_i$s are the side lengths of the box.\\
Let $\chi$ be a multiplicative character on $\mathbb{F}_{p^n}$. Throughout the paper, implied $O$ constant may depend on $n$ and $\varepsilon$ unless otherwise stated.
\par
The study of multiplicative character sums over finite fields is a central topic in analytic number theory, with classical results going back to Burgess. A natural problem is to understand the extent of cancellation in such sums when restricted to structured subsets such as boxes in $\mathbb{F}_{p^n}$.
\par
Suppose, we have ordered the side lengths of the box in an increasing order as $H_1\leq H_2\leq \cdots\leq H_n$. In this paper, we establish a Burgess-type bound for multiplicative character sums over general boxes in $\mathbb{F}_{p^n}$ under a volumetric condition where the sides are of length at least $(p^{n(1/4+\varepsilon)}/H_{n-1}H_{n})^{1/(n-2)}$. Finally, we obtain a nontrivial cancellation of character sums under the conditions expressed before. Our result generalizes a theorem of Konyagin \cite{Kon}, who obtained nontrivial cancellation under the stronger assumption that every side length is at least $p^{1/4+\varepsilon}$. It also extends the work of Gabdullin \cite{GB}, who established analogous results only in the cases $n=2$ and $n=3$, to finite field extensions of arbitrary degree. Our main theorem in this paper is the following  
\begin{theorem}[Main Theorem]\label{MT}
Let $n > 2$, and let $\chi$ be a nontrivial multiplicative character of $\mathbb{F}_{p^n}$. Given $\varepsilon>0$, there is a $\delta_n(\varepsilon)$ such that if
\begin{equation}\label{mncon}
(p^{n(1/4+\varepsilon)}/H_{n-1}H_{n})^{1/(n-2)}\le H_1\leq H_2\leq \cdots\leq H_{n-2}\le H_{n-1}\le H_n \le \sqrt{p/2}.
\end{equation}
Then the following holds:
\begin{enumerate}
\item[(i)]
\[
\left|\sum_{x\in B} \chi(x)\right|
\ll_{n,\varepsilon} 
|B|\, p^{-\delta_n(\varepsilon)},
\]
where
\[
\delta_n(\varepsilon)=\varepsilon^2\frac{2n^2-n}{16n^2-1}.
\]
\item[(ii)] Suppose either $(n,6)=1$, or that for every divisor $r\mid n$ with $r=2,3$, the restriction of $\chi$ to $\mathbb{F}_{p^{n/r}}$ is nontrivial, that is,
\[
\chi\!\restriction_{\mathbb{F}_{p^{n/r}}} \neq \chi_0.
\]
Then the conclusion of\/ {\rm (i)} remains valid without the condition $H_n \le \sqrt{p/2}$.
\end{enumerate}
\end{theorem}

\medskip

We begin with some historical background. In the case $n=1$, Burgess's bound \cite{B1} remains the strongest result to date. It asserts that for every $\varepsilon>0$ there exists $\delta>0$ such that
\[
\left|\sum_{x=N+1}^{N+H}\chi(x)\right|\ll_{\varepsilon} p^{-\delta}H
\]
whenever $H\geq p^{1/4+\varepsilon}$. In subsequent work \cite{B2}, Burgess also obtained an analogue of this bound for $n=2$ in the case of certain special bases. Karatsuba \cite{K} later extended these ideas to general finite fields $\mathbb{F}_{p^n}$, under the assumption that the basis is generated by a root of an irreducible polynomial of degree $n$ over $\mathbb{F}_p$.

For arbitrary bases, the first nontrivial bounds were obtained by Davenport and Lewis \cite{DL}. They studied character sums over $\mathbb{F}_{p^n}$ of the form
\[
\sum_{x_1\in I_1,\ldots,x_n\in I_n} \chi\!\left(x_1\omega_1+\cdots+x_n\omega_n\right),
\]
where $\{\omega_1,\ldots,\omega_n\}$ is a fixed basis of $\mathbb{F}_{p^n}$ over $\mathbb{F}_p$ and $I_i$ are intervals in $\mathbb{Z}$. They proved that 
\[
\sum_{x \in B}\chi(x )
=O_{\varepsilon}\!\left(H_1\cdots H_n\,p^{-\delta(\varepsilon)}\right),
\]
provided that $H_i=p^{\rho_n+\varepsilon}$ for all $i$, where
\[
\rho_n=\frac12-\frac{1}{2(n+1)}.
\]
This result was later improved by M. C. Chang \cite{C1}, who proved that
\[
\left|\sum_{x \in B}\chi(x )\right|
\ll_{n,\varepsilon} p^{-\varepsilon^2/4}|B|,
\]
provided that $|B|\geq p^{n(2/5+\varepsilon)}$, $\chi\!\restriction_{\mathbb{F}_{p^{n/2}}}\neq \chi_0$ and $n$ is even. In particular, this produces non-trivial cancellation whenever $H_i\geq p^{2/5+\varepsilon}$ for all $1\leq i\leq n$. Furthermore, this improves upon the Davenport–Lewis bound for $n\geq5$. Chang \cite{Cha} also established Burgess-strength bounds in each coordinate for $n=2$, assuming $H_1,H_2\geq p^{1/4+\varepsilon}$. 

Subsequently, S. Konyagin generalized Chang's two-dimensional result to arbitrary finite fields of degree $n$ over $\mathbb{F}_p$, establishing Burgess-strength estimates in every coordinate.

\begin{theorem}[Konyagin \cite{Kon}]\label{kon}
Let $\varepsilon>0$ and suppose $H_i\geq p^{1/4+\varepsilon}$ for all $1\leq i\leq n$. Then
\[
\left|\sum_{x \in B}\chi(x )\right|
\ll_{n,\varepsilon} p^{-\varepsilon^2/2}|B|.
\]
\end{theorem}

While this theorem provides strong cancellation, it requires each interval to be of Burgess length, precluding situations in which some intervals are short while others are long. This limitation was partially overcome by M. R. Gabdullin \cite{GB}, who generalized Konyagin's result for $n=2,3$ under a weaker product condition.
\begin{theorem}[Gabdullin \cite{GB}]\label{gab}
Let $n\in\{2,3\}$, let $\chi$ be a nontrivial multiplicative character of $\mathbb{F}_{p^n}$, and let $\varepsilon>0$. Suppose that
\[
|B|\ge p^{n(1/4+\varepsilon)}, \qquad  H_1\le \cdots \le H_n.
\]
Then
\[
\left|\sum_{x \in B}\chi(x )\right|
\ll_{n,\varepsilon} |B|\, p^{-\varepsilon^2/12},
\]
provided that $\chi|_{\mathbb{F}_p}$ is nontrivial, and otherwise
\[
\left|\sum_{x \in B}\chi(x )\right|
\ll_{n,\varepsilon} |B|\, p^{-\varepsilon^2/12} + |B\cap \omega_n \mathbb{F}_p|.
\]
\end{theorem}
In the main theorem, Theorem~\ref{MT}, taking $n=3$ we recover Theorem \ref{gab} of Gabdullin for $n=3$ and by a possible modification of Theorem~\ref{MT} for $n=2$ we can also recover Theorem \ref{gab} for $n=2$ as well.  On the other hand, suppose $H_i\geq p^{1/4+\varepsilon}$ for all $1\leq i\leq n$ as in Theorem \ref{kon}. Therefore, $(p^{n(1/4+\varepsilon)}/H_{n-1}H_{n})^{1/(n-2)}\leq p^{1/4+\varepsilon}$ and the condition \eqref{mncon} in the Theorem is satisfied. Following this one can apply Theorem \ref{MT} and obtain a power saving upper bound of the character sum. Thus, Theorem~\ref{MT} simultaneously generalizes both results, as mentioned above. 

Next, we record the following remarks on our main Theorem \ref{MT}:
\begin{remark}
The condition \eqref{mncon} of Theorem \ref{MT} is stronger than $|B|\geq p^{n(1/4+\varepsilon)}$ of Theorem \ref{gab}. In other words, Gabdullin's theorem can be applied to a larger class of boxes, as there is more flexibility in the side lengths of the boxes.
\end{remark}
\begin{remark}
If $\chi$ is induced by the trivial character on a proper subfield 
$\mathbb{F}_{p^{n/d}} \subsetneq \mathbb{F}_{p^n}$ and 
$B \subseteq \mathbb{F}_{p^{n/d}}$, then
\[
\sum_{x\in B} \chi(x)=|B|,
\]
and so no nontrivial cancellation occurs. Since $|B|> p^{n/4}$, this forces $d< 4$. This explains the necessity of imposing a non-triviality condition on the restriction of $\chi$ to proper sub-fields in Theorem~\ref{MT}(ii).
\end{remark}
\begin{remark}
Arguing as in the proof of Theorem~\ref{MT}, one can also obtain nontrivial cancellation for $|B|\geq p^{n/3+\varepsilon}$
under the weaker condition that 
\[
\chi\!\restriction_{\mathbb{F}_{p^{n/2}}} \neq \chi_0.
\]
In view of Remark~2, this condition is necessary in this range. Indeed, if 
\(\chi\!\restriction_{\mathbb{F}_{p^{n/2}}} = \chi_0\), then no cancellation of the character occurs if the box $B$ is contained in $\mathbb{F}_{p^{n/2}}$.
\end{remark}

\subsection{Outline of the proof of main theorem.} The proof of the Main Theorem closely follows the approach of \cite{GB}, employing techniques from \cite{C1}, \cite{K} and \cite{Kon} to estimate character sums over boxes in finite fields whose side lengths are at most $\sqrt{p/2}$. Our argument begins with the classical additive shifting technique $x\mapsto x+yz$ introduced by Burgess \cite{B2}. Applying H\"older's inequality twice to the resulting sum of shifted characters produces three terms: one is related to the multiplicative energy of the box, while the remaining terms can be bounded using Weil's bound and trivial estimates. Consequently, achieving a nontrivial saving for the character sum reduces to obtaining a bound on the multiplicative energy of a set in $\mathbb{F}_{p^n}$ which are established in Lemmas \ref{l1} and Lemma \ref{l2}. More specifically, we employ tools from the geometry of numbers, in particular the concepts of successive minima and Minkowski's second theorem, together with an auxiliary lemma that provides upper and lower bounds for the product of successive minima which turns out to be an essential ingredient in our energy estimate.

\medskip

On the other hand, for boxes whose side lengths exceed $p^{1/2+\varepsilon}$, we are again able to derive a non-trivial upper bound for the associated character sum where the character is nontrivial in proper subfields of $\mathbb{F}_{p^n}$. The key ingredient in this case is the Proposition~\cite{Katz} of N.\,M. Katz concerning the complete character sums. Under the standing assumption that the larger side length of the box exceeds $p^{1/2+\varepsilon}$, we obtain a non-trivial saving for the character sum along that direction. Estimating the remaining summations trivially then yields an overall non-trivial saving over the volume of the box. In addition, the Pólya--Vinogradov inequality in finite fields is invoked to complete the argument in this setting. For further details, see \textbf{Case~3} in the proof of the main theorem.

\section{Notation and Preliminaries}
We begin by defining multiplicative energy of a set. In the multiplicative setting, the energy counts the number of solutions to the equation $ab = cd$ with variables drawn from a given set.
\begin{Definition}[Multiplicative energy]
Let $A$ be a finite subset of a commutative ring (or field).
The multiplicative energy of $A$ is defined by
\[
E(A)
=
\bigl|\{(a,b,c,d)\in A^4 : ab = cd\}\bigr|.
\]
\end{Definition}
\begin{Definition}[Successive minima]
Let $\Lambda$ be a lattice in $\mathbf{R}^d$ of rank $k$, and let
$B$ be a convex body in $\mathbf{R}^d$. We define the $j-$th successive minima $\lambda_j = \lambda_j(B, \Lambda)$ for
$1\leq  j \leq k$ of $B$ with respect to $\Lambda$ as
\[
\lambda_j := \inf \left\{ \lambda > 0 \;:\; 
\lambda B \text{ contains } j \text{ linearly independent elements of } \Lambda 
\right\}.
\]
Note that $0 < \lambda_1 \leq\ldots \leq \lambda_k < \infty$.
\end{Definition}
\begin{Proposition}[Minkowski’s second theorem]\label{Mkow}
    Let $\Lambda$ be a lattice of full rank in $\mathbf{R}^d$, and let $B$ be a closed convex centrally symmetric in $\mathbf{R}^d$ , with successive minima $0 < \lambda_1 \leq \ldots \leq \lambda_d$. Then there exist $d$ linearly independent vectors $v_1,\ldots,v_d \in\mathbf{R}^d$
with the following properties:
\begin{itemize}
\item For each $1 \le j \le d$, the vector $v_j$ lies on the boundary of $\lambda_j B$, and $\lambda_j B$ does not contain lattice points of $\Lambda$ outside the span of $v_1,\ldots,v_{j-1}$.
\item \[
\frac{2^d}{d!}\,\operatorname{\mu}(\mathbf{R}^d/\Lambda)
\le
\lambda_1 \lambda_2 \cdots \lambda_d \, \operatorname{\mu}(B)
\le
2^d \operatorname{\mu}(\mathbf{R}^d/\Lambda).\]
where $\mu$ is the Lebesgue measure.
\end{itemize}
\end{Proposition}
\begin{proof}
    See \cite[Theorem 3.30]{TV} for a proof.
\end{proof}
We recall the following lemma which provides an estimate for the number of points of the lattice $\Lambda$ which lies inside the convex body $B$ with respect to the successive minima defined above.
\begin{Lemma}\label{LCE}
     Let $\Lambda$ be a lattice of full rank in
$\mathbf{R}^d,$ and let $B$ be a symmetric convex body in $\mathbf{R}^d$ as before, with successive minima $0 < \lambda_1 \leq \ldots \leq \lambda_d$. Then
\begin{equation*}
\left|\Lambda\cap B\right|\ll \prod_{j=1}^d\max\{1,\lambda_j^{-1}\}    
\end{equation*}
\end{Lemma}
\begin{proof}
    This is a standard lattice counting inequality. See  \cite[Proposition 2.1]{BHW} or  \cite[Exercise 3.5.6]{TV} for details of the proof.
\end{proof}
We shall use the above lemma to obtain lower bounds for our successive minima, which will be required to get an effective estimate for the multiplicative energy of the desired set. 

\medskip

The polar lattice $L^*$ of a lattice $L\subset \mathbb{R}^d$ and the polar body $D^*$ of a convex symmetric body $D\subset \mathbb{R}^d$ are defined by
\begin{equation}\label{LD}
L^*=\{x\in \mathbb{R}^d:\ \langle x,y\rangle\in \mathbb{Z} \text{ for all } y\in L\}, 
\qquad 
D^*=\{x\in \mathbb{R}^d:\ \langle x,y\rangle\le 1 \text{ for all } y\in D\}.
\end{equation}
where, the usual inner product is defined as $\langle x,y\rangle:= x_1 y_1+x_2y_2+\cdots+x_dy_d$ with $x=(x_1,\ldots,x_d)$ and $y=(y_1,\ldots,y_n).$

\medskip
The following result, due to Mahler, relates the successive minima of $L$ with respect to $D$ and those of $L^*$ with respect to $D^*$.

\begin{Lemma}\label{Mah}
Let $L\subset \mathbb{R}^d$ be a lattice and let $D\subset \mathbb{R}^d$ be a symmetric convex body. Let $\lambda_1,\ldots,\lambda_d$ and $\lambda_1^*,\ldots,\lambda_d^*$ denote the successive minima of $L$ with respect to $D$ and of $L^*$ with respect to $D^*$, respectively. Then, for each $1\le i\le d$, we have
\[
\lambda_i \lambda_{d-i+1}^* \ll_d 1.
\]
\end{Lemma}

\begin{proof}
See \cite[Proposition 3.6]{Ban}.
\end{proof}
\medskip

To estimate the character sums arising in our arguments, we shall rely on classical bounds due to Weil. These results provide square-root cancellation for multiplicative character sums over finite fields and play a central role in controlling higher moments of short character sums. \\
We first recall Weil’s bound for character sums of polynomial arguments, and then derive a consequence adapted to incomplete sums over intervals, which will be used repeatedly in what follows.
\begin{Proposition}\label{WE}[Weil] Let $\chi$ be a multiplicative character of $\mathbb{F}_{p^n}$ of order $d>1$. Assume that a polynomial $f\in \mathbb{F}_{p^n}[x]$ has $m$ distinct roots and is not a $d-$th power. Then
\[\left|\sum_{x\in \mathbb{F}_{p^n}}\chi(f(x))\right|\leq (m-1)p^{n/2}.\]
\end{Proposition}
\begin{proof}
See  \cite[Theorem 2C']{Sch}.
\end{proof}
\begin{Lemma}\label{WE2}
Let $\chi$ be a multiplicative character of $\mathbb{F}_{p^n}$ of order $d>1$, and let $I \subset \mathbb{Z}$ be an interval. Then for every integer $r \ge 1$, we have
\[
\sum_{u \in \mathbb{F}_{p^n}}
\left|\sum_{z \in I} \chi(u+z)\right|^{2r}
\ll
2r\, p^{n/2} |I|^{2r}
\;+\;
r^{2r}p^n\, |I|^r.
\]
The implied constant is absolute.
\end{Lemma}
\begin{proof}
Expanding the $2r$-th moment gives
\[
\sum_{u \in \mathbb{F}_{p^n}}
\left|\sum_{z \in I} \chi(u+z)\right|^{2r}
=
\sum_{z_1,\ldots,z_{2r}\in I}
\sum_{u \in \mathbb{F}_{p^n}}
\chi\!\left(
\prod_{i=1}^{r}(u+z_i)
\prod_{i=r+1}^{2r}(u+z_i)^{-1}
\right).
\]
\[
=\sum_{z_1,\ldots,z_{2r}\in I}
\sum_{u \in \mathbb{F}_{p^n}}
\chi\!\left(
\prod_{i=1}^{r}(u+z_i)
\prod_{i=r+1}^{2r}(u+z_i)^{p^n-2}
\right).
\]
The last equality follows as $u^{p^n-2}$ is the inverse of an element $u\in \mathbb{F}_{p^n}^*$.
Fix a tuple $\mathbf{z}=(z_1,\ldots,z_{2r}) \in I^{2r}$ and define the polynomial function
\[
F_{\mathbf{z}}(u)
=
\prod_{i=1}^{r}(u+z_i)
\prod_{i=r+1}^{2r}(u+z_i)^{p^n-2}.
\]

We classify tuples into two types.

\medskip
\noindent
\textbf{(i) Bad tuples.}
We call $\mathbf{z}$ a \emph{bad tuple} if the polynomial $F_{\mathbf{z}}(u)$ becomes a perfect $d$-th power in $\mathbb{F}_{p^n}[u]$. This can only happen if every distinct value among $z_1,\ldots,z_{2r}$ occurs with multiplicity at least two in this collection and the collection contains at most $r$ distinct elements. Hence, the number of bad tuples is at most
$ r^{2r}|I|^r.$ For each bad tuple, we use the trivial bound for character sum which is $p^n$. So, the contribution for this case is dominated by $p^n r^{2r}|I|^r$.

\medskip
\noindent
\textbf{(ii) Good tuples.}
For all remaining tuples (called \emph{good tuples}), the polynomial $F_{\mathbf{z}}(u)$ is not a perfect $d$-th power in $\mathbb{F}_{p^n}[u]$. Hence, by Proposition \ref{WE}, the character sum corresponding to these tuples is bounded by $2r\, p^{n/2}$. Since there are at most $|I|^{2r}$ tuples in total, the contribution of good tuples is
$ 2r\, p^{n/2} |I|^{2r}$.

\medskip
\noindent
Combining the contributions of bad and good tuples completes the proof.
\end{proof}
Let us recall P\'{o}lya-Vinogradov inequality over finite fields.
\begin{Proposition}[P\'{o}lya-Vinogradov over $\mathbb{F}_{p^d}$]\label{PV}
Let $\chi$ be a nontrivial $(\neq \chi_0)$ multiplicative character on $\mathbb{F}_{p^n}$. Let
\[
B=\{x_1\omega_1+\cdots+x_n\omega_n:\; x_i\in I_i\subset \mathbb{F}_p\}
\]
be a box where $\{\omega_1,\ldots, \omega_n\}$ is a basis of $\mathbb{F}_{p^n}$ over $\mathbb{F}_p$. Then for any $a\in\mathbb{F}_{p^n}$,
\[
\sum_{x\in B}\chi(x+a)\ll p^{n/2}(\log p)^n.
\]
\end{Proposition}

\begin{proof} We shall briefly sketch the proof as it follows a standard Fourier theoretic argument. Let $\mathbf{x}=(x_1,\ldots, x_n)$ for $x=x_1\omega_1+x_2\omega_2+\cdots+x_d\omega_n$. Expanding each interval $I_i$ into additive characters, we write
\[
1_B(x)=\sum_{\mathbf{t}\in\mathbb{F}_p^n} c_{\mathbf{t}}\,\psi(\mathbf{t}\cdot \mathbf{x}),
\]
where $\mathbf{t}.\mathbf{x}=(t_1x_1+\cdots+t_nx_n)$, $\psi(x)=e(\frac{2\pi i x}{p})$, $c_{\mathbf{t}}=\frac{1}{p^n}\sum_{\mathbf{x}\in B}\psi(-\mathbf{t}\cdot \mathbf{x})$ and $\sum_{\mathbf{t}}|c_{\mathbf{t}}|\ll (\log p)^n$. Hence
\[
\sum_{x\in B}\chi(x+a)
=
\sum_{\mathbf{t}} c_{\mathbf{t}} \sum_{x\in\mathbb{F}_{p^n}}\chi(x+a)\psi_{\mathbf{t}}(x).
\]
After substituting $x+a$ with $x$, the inner sum becomes a Gauss sum, which vanishes for $\mathbf{t}=0$ and is $O(p^{n/2})$ otherwise (since, $\chi\neq \chi_0$). The result follows.
\end{proof}
\section{Estimates for Multiplicative Energy}
In this section, we derive an essentially sharp upper bound for the multiplicative energy of the set \( B \), which constitutes a central ingredient in our small character sum estimates. Using methods from additive combinatorics, Chang~\cite[Proposition~5]{C1} proved that
\[
E(B) \ll |B|^{11/4}\log p
\]
whenever the side lengths satisfy \( H_i < (\sqrt{p}-1)/2 \). Subsequently, applying techniques from the geometry of numbers, Konyagin~\cite[Lemma~1]{Kon} established the stronger estimate
\[
E(B) \ll_n |B|^2 \log p
\]
for boxes with equal side lengths, namely \( H_1=\cdots=H_n \leq \sqrt{p} \).
Later, Gabdullin \cite[Key Lemma]{GB} extended Konyagin’s lemma to boxes with unequal side lengths in the cases \( n=2 \) and \( n=3 \). In the present work, we generalize the result of Konyagin to arbitrary dimension \( n \in \mathbb{N} \) for boxes with at least $(n-2)$ sides of comparable length, with each of these side lengths lying in the same dyadic interval. In particular, our result recovers Gabdullin's lemma as a special case when $n=3$.
\begin{Lemma}[Main Lemma]\label{KL}
Let $n\in \mathbb{N}\setminus\{1,2\}$ and suppose $H\leq H_1\leq H_2\leq \ldots\leq H_{n-2}\leq H_{n-1}\leq H_n<\sqrt{p/2},\;H_{n-2}<2H.$ Then we have 
\[E(B)\ll |B|^2 (\log p)^n.\]
\end{Lemma}
\begin{proof}
The proof starts with the same idea as in the proof of the key lemma of \cite{GB}.  
As in the proof of the key lemma, we define
\[
Z' \;=\; \frac{B\setminus\{0\}}{B\setminus\{0\}}
    \;=\;
    \{\, z\in \mathbb{F}_{p^{n}} : xz = y \text{ for some } x,y \in B\setminus\{0\} \,\}.
\]
and $f(z)$ as the number of solutions $(x,y)\in B^{2}$ to the relation $xz=y$ where $z\in Z'$.

\medskip
We also define
\[
B_{0}
    = 
    \left\{
        \sum_{i=1}^{n} x_{i}\omega_{i} \;:\;
        -H_{i}\leq x_{i}\leq H_{i},\; 1\leq i \leq n
    \right\},
    \qquad
Z = \frac{B_{0}\setminus\{0\}}{B_{0}\setminus\{0\}}.
\]
In complete analogy with $f(z)$, set
\[
f_{0}(z)
    = 
    |\{\, (x,y)\in B_{0}^{2} : xz = y \,\}|.
\]
As in the Key lemma \cite{GB} the multiplicative energy can be estimated as the following
\[
E(B)
    \;\leq\;
    3|B|^{2} + \sum_{z\in Z} f_{0}(z)^{2}.
\]
Proceeding as in \cite[Equation (2.2)]{GB}, it remains to estimate the quantity

\[
L_{1}
    = 
    \sum_{z\in Z\setminus\mathbb{F}_{p}} f_{0}(z)^{2},
    \qquad
L_{2}
    =
    \sum_{z\in \mathbb{F}_{p}^*} f_{0}(z)^{2}.
\]
The following two lemmas provide the desired bounds for Lemma \ref{KL}.
\begin{Lemma}\label{l1}
We have
\[
L_{1} \;\ll\; |B|^{2}\log p.
\]
\end{Lemma}

\begin{Lemma}\label{l2}
We have
\[
L_{2} \;\ll\; |B|^{2}(\log p)^{n}.
\]
\end{Lemma}

\end{proof}
We shall establish the above two estimates in the following subsections.
\subsection{Proof of Lemma \ref{l1}}
Fix $z\in Z$ and consider the lattice $\Lambda_{z}\subset \mathbb{Z}^{2n}$ defined by 
\[
\Lambda_{z}
=\Big\{(x_1,\ldots,x_n,y_1,\ldots,y_n)\in\mathbb{Z}^{2n}:\;
z\textstyle\sum_{i=1}^{n}x_i\omega_i
=\sum_{i=1}^{n}y_i\omega_i\Big\}.
\]
For any fixed $(x_1,\ldots,x_n)\in\mathbb{Z}^{n}$, the defining relation uniquely determines each of the coordinates $y_i$ modulo $p$. Consequently, for any $M\ge 1$,
\[
\big|\Lambda_{z}\cap[-M,M]^{2n}\big|
=\frac{(2M)^{2n}}{p^{\,n}}(1+o(1)) \qquad (M\to\infty),
\]
and therefore
\[
\mu(\mathbb{R}^{2n}/\Lambda_{z})=p^{n}.
\]
Define the symmetric convex body
\[
D=\Big\{(x_1,\ldots,x_n,y_1,\ldots,y_n)\in\mathbb{R}^{2n}:\; |x_i|,|y_i|\le H_i,\ 1\le i\le n\Big\},
\]
and let $\lambda_1(z)\le\cdots\le\lambda_{2n}(z)$ denote the successive minima of $D$ with respect to $\Lambda_z$. Here it is essential to note that $f_0(z)=|\Lambda_z\cap D|$ by the definition of $f_0(z)$. Since $z\in Z$, the body $D$ contains a nonzero lattice point, and hence
\[
\lambda_1(z)\le 1.
\]
By Proposition \ref{Mkow} (Minkowski's second theorem),
\begin{equation}\label{eq:minkowski-product}
\prod_{i=1}^{2n}\lambda_i(z)
\gg
\frac{\operatorname{\mu}(\mathbb{R}^{2n}/\Lambda_z)}{\operatorname{\mu}(D)}
\gg
p^{n}|B|^{-2}.
\end{equation}
Furthermore, by the standard lattice--counting inequality, Lemma \ref{LCE},
\begin{equation}\label{eq:f0-bound}
f_0(z)=|\Lambda_z\cap D|
\;\ll\;
\prod_{i=1}^{2n}\max\{1,\lambda_i(z)^{-1}\}.
\end{equation}
We now establish a uniform lower bound for the first minimum. We claim that
\begin{equation}\label{eq:lambda1-lower}
\lambda_1(z)\ge H_{n-1}^{-1}.
\end{equation}
Suppose instead that $\lambda_1(z)<H_{n-1}^{-1}$. Then $\lambda_1(z)D$ contains a nonzero lattice vector of the form
\[
(0,\ldots,0,u_n,0,\ldots,0,u_{2n})\in\Lambda_z,
\]
and therefore
\[
|u_n|<H_n\lambda_1(z) < H_nH_{n-1}^{-1},
\qquad
|u_{2n}|<H_n\lambda_1(z) < H_nH_{n-1}^{-1}.
\]
From the defining relation of $\Lambda_z$ we obtain
\[
z\,u_n\omega_n=u_{2n}\omega_n,
\]
which implies
\[
z=\frac{u_{2n}}{u_n}\in\mathbb{F}_{p},
\]
contradicting the assumption that $z\notin\mathbb{F}_p$. Thus \eqref{eq:lambda1-lower} follows.\\
Now, for every $z\in Z\setminus \mathbb{F}_{p}$ we have $\lambda_{1}(z)\le 1$, and
\[
\lambda_{2n}(z)\ge\lambda_{2n-1}(z)\ge\cdots\ge \lambda_{1}(z)\ge H_{n-1}^{-1}.
\]
For $1\le j\le J:=[\log_2 H_{n-1}]+1$, define the dyadic classes
\[
Z_j=\Big\{z\in Z\setminus \mathbb{F}_{p}:\ 2^{\,j-1}\le H_{n-1}\lambda_1(z)<2^{\,j}\Big\}.
\]
For each $z\in Z_j$, the minimal vector $u\in \lambda_1(z)D\cap \Lambda_z$ uniquely determines $z$. Indeed, write 
\[
u=(u_1,\ldots,u_{2n}),
\qquad
x=\sum_{i=1}^{n}u_i\omega_i,\qquad
y=\sum_{i=1}^{n}u_{n+i}\omega_i,
\]
so that the lattice relation implies $zx=y$ in $\mathbb{F}_{p^n}$, and therefore
$z=yx^{-1}.$
Thus, the map $z\mapsto u$ is injective on each $Z_j$.  
Consequently,
\[
|Z_j|
\le 
\big| 2^{\,j}H_{n-1}^{-1}D \cap \mathbb{Z}^{2n}\big|.
\]
Now after setting $j'=\log_2(H_{n-1}/H_1)$ we obtain that
\begin{equation}\label{Z}
    |Z_j|\le \big| 2^{\,j}H_{n-1}^{-1}D\, \cap\, \mathbb{Z}^{2n}\big|\ll \prod_{i=1}^n\max\{1,H_i2^jH_{n-1}^{-1}\}^2\leq 
    \begin{cases}
        2^{2(n-1)j}H_2^2\ldots H_{n-2}^2\,H_n^2\,H_{n-1}^{-2(n-2)} \text{ if }1\leq j< j'\\
        2^{2nj}|B|^2H_{n-1}^{-2n}\hspace{1.17in}\text{ if } j'\leq j\leq J.\\
    \end{cases}
\end{equation}
Further set $s=s(z)=\max\{j:\lambda_j\leq 1\}$ and $Z^s=\{z\in Z\setminus \mathbb{F}_p:s(z)=s\}.$ Recalling the definition of $L_1$, we get
$$L_1 \leq\sum_{s=1}^{2n}\sum_{z \in Z^s}f_0^2(z).$$
Let us define $Z_j^s=Z^s\cap Z_j$. Then
$$\sum_{z\in Z^s}f_0^2(z)\leq \sum_j\sum_{z\in Z_j^s}f_0^2(z).$$
\textbf{Case 1: Let $1 \le s < n$.}
We have
\[
f_0(z) \ll \lambda_1^{-1}\lambda_2^{-1}\cdots\lambda_s^{-1}.
\]
Using \eqref{Z} and the bound
\[
\lambda_{2n}^{-1} \le \lambda_{2n-1}^{-1} \le \cdots \le \lambda_1^{-1}
\ll 2^{-j} H_{n-1}, \qquad z \in Z_j,
\]
we obtain
\begin{equation}
\begin{split}
\sum_{j}\sum_{z \in Z_j^s} f_0^2(z)
= {} & \sum_j \sum_{z \in Z_j^s}
\lambda_1^{-2}\lambda_2^{-2}\cdots\lambda_s^{-2} \\
= {} & \sum_{1 \le j < j'}
|Z_j^s|\,\lambda_1^{-2}\cdots\lambda_s^{-2}
+ \sum_{j' \le j \le J}
|Z_j^s|\,\lambda_1^{-2}\cdots\lambda_s^{-2} \\
\ll {} &
\sum_{1 \le j < j'}
\left(
2^{2(n-1)j}
H_2^2 \cdots H_{n-2}^2 H_n^2
H_{n-1}^{-2(n-2)}
\right)
\left(2^{-j}H_{n-1}\right)^{2s} \\
& \quad
+ \sum_{j' \le j \le J}
2^{2nj} |B|^2 H_{n-1}^{-2n}
\left(2^{-j}H_{n-1}\right)^{2s} \\
\ll {} &
\left(\frac{H_{n-1}}{H_1}\right)^{2(n-1-s)}
H_2^2 \cdots H_n^2 H_{n-1}^2
H_{n-1}^{-2(n-s)}
+ 2^{4(n-s)} |B|^2 \\
\ll {} &\hspace{.2in} |B|^2 .
\end{split}
\end{equation}
\textbf{Case 2: Let $s = n$.}
Proceeding as in Case~1, it is straightforward to obtain
\[
\sum_j \sum_{z \in Z_j^n} f_0^2(z)
\ll
\sum_j \sum_{z \in Z_j^n}
\lambda_1^{-2}\lambda_2^{-2}\cdots\lambda_n^{-2}
\ll |B|^2 \log p .
\]
\textbf{Case 3: Consider the other cases $n<s\leq 2n.$}

\noindent\medskip
Let us first begin with $s=2n$. From the equation \eqref{eq:minkowski-product} and the boundedness assumption on the side lengths of the boxes, we deduce
\[\sum_{z\in Z^s}f_0^2(z)\ll \sum_{z\in Z}|B|^4p^{-2n}\ll |B|^2|B|^4p^{-2n}\ll |B|^2p^{4n/2}p^{-2n}\leq |B|^2.\]
Now, it suffices to show a similar bound for $n<s<2n$. 

\noindent\medskip
We begin by recalling the definition of the polar (dual) lattice (see \eqref{LD}):
\begin{equation*}
\Lambda_z^*=\Bigl\{(u_1,\ldots,u_n,u_{n+1},\ldots,u_{2n})\in \mathbb{R}^{2n}:\ 
\sum_{i=1}^n u_i x_i+\sum_{i=1}^n u_{i+n} y_i \in \mathbb{Z} 
\ \text{for all } (x_1,\ldots,x_n,y_1,\ldots,y_n)\in \Lambda_z \Bigr\}.
\end{equation*}
Since $\Lambda_z \supseteq p\mathbb{Z}^{2n}$, it follows that $\Lambda_z^* \subseteq p^{-1}\mathbb{Z}^{2n}$.
We also define the polar body $D^*$ (cf. \eqref{LD}) by
\begin{equation*}
D^{*}
=
\Bigl\{
(u_1,\ldots,u_n,v_1,\ldots,v_n)\in \mathbb{R}^{2n} :
\sum_{i=1}^n |u_i x_i|
+ \sum_{i=1}^n |v_i y_i|
\le 1
\ \text{for all } (x_1,\ldots,x_n,y_1,\ldots,y_n)\in D
\Bigr\}.
\end{equation*}
Clearly,
$$D^*=\left\{(u_1,\ldots,v_n)\in \mathbb{R}^{2n}:\sum_{i=1}^n(|u_i|+|v_i|)H_i\leq 1\right\}.$$
Let $\lambda_1^*=\lambda_1(z)^*$ be the first successive minimum of the set $D^*$ with respect to $\Lambda^*_z$. By Lemma \ref{Mah}, we have 
\begin{equation}\label{Res}
    \lambda_1^*\lambda_{2n}\ll 1.
\end{equation}
Considering \eqref{eq:minkowski-product} and \eqref{eq:f0-bound} for $z\in Z_s$ where, $n<s<2n$ we obtain
\begin{equation}\label{f_0-boundnew}
 \begin{split}
     f_0(z)\ll \prod_{i=1}^s\lambda_i^{-1}(z)\leq \lambda_{2n}^{(2n-s)}\prod_{i=1}^{2n}\lambda_i^{-1}(z).
 \end{split}   
\end{equation}
Now from \eqref{eq:minkowski-product} and \eqref{Res} we further get
\begin{equation}\label{f_0-boundnew2}
    f_0(z)\ll \lambda_{2n}^{(2n-s)}|B|^2p^{-n}\ll (\lambda_1^*)^{-(2n-s)}|B|^2p^{-n}.
\end{equation}
Again we split the above estimation in two cases depending on the value of $\lambda_1^*(z)$.

\medskip

\textbf{Case 1: $\lambda_1^*(z)\geq 1$.}\\
In this case, equation \eqref{f_0-boundnew2} is actually bounded by $|B|^2p^{-n}$. So, we can estimate $\sum_{z \in Z^s} f_0(z)^2$ similarly as in case $s=2n$.

\medskip

\textbf{Case 2: $\lambda_1^*(z)< 1$.}\\
Suppose, in addition, if $\lambda_1^*(z)<H_1p^{-1}$ then one checks from the inclusion $\Lambda^*_z\subseteq p^{-1}\mathbb{Z}^{2n}$ that $\lambda_1^*D^*\cap \Lambda^*_z=\{0\}$ which is surely a contradiction. So, $\lambda_1^*(z)\geq H_1p^{-1}$.\\
Set
\[Z'_j=\left\{z\in Z: 2^{j-1}\leq \frac{p\lambda^*_1(z)}{H_1}<2^j\right\}\]
where, $j=1,2,\ldots,\left[\log_2 \frac{p}{H_1}\right]+1$.
We  now claim that the non-zero vector $u\in \lambda^*_1(z)D^*\cap\Lambda^*_z$ uniquely corresponds to an element $z\in Z_j'$.\\
Suppose, on the contrary, 
\[u=(u_1/p,\ldots,u_n/p,v_1/p,\ldots,v_n/p)\in \Lambda^*_{z'}\cap\Lambda^*_{z''}\]
where, $z'\neq z''$ and $u_i,v_i\in Z_j$. Now, as $u\in \lambda^*_1(z)D^*$ and $\lambda^*_1(z)<\frac{H_1}{p}2^j$ the coordinates of the vector $u$ follow $\sum_{i=1}^n|u_i|+\sum_{i=1}^n|v_i|<2^j$. Take an arbitrary element 
\[x=\sum_{i=1}^nx_i\omega_i\in \mathbb{F}_{p^n}\] and set
\[
\begin{cases}
    y'=xz'=\sum_{i=1}^ny_i'\omega_i\\
    y''=xz''=\sum_{i=1}^ny_i''\omega_i.
\end{cases}
\]
Then $(x_1,x_2,\ldots,x_n,y_1',\ldots,y_n')\in \Lambda_{z'}$ and $(x_1,x_2,\ldots,x_n,y_1'',\ldots,y_n'')\in \Lambda_{z''}$ and by the definition of polar set:
\begin{equation}
\begin{cases}
\sum_{i=1}^n\frac{x_iu_i}{p}+\sum_{i=1}^n\frac{y_i'v_i}{p}\in \mathbb{Z}\\
\sum_{i=1}^n\frac{x_iu_i}{p}+\sum_{i=1}^n\frac{y_i''v_i}{p}\in \mathbb{Z}.
\end{cases}
\end{equation}
By subtracting the above two, it implies \[\sum_{i=1}^n(y_i'-y_i'')v_i\equiv 0 \pmod{p}.\]Also note that $(y_i'-y_i'')$ can be chosen arbitrarily and thus $v_i \equiv 0 \pmod{p}
\quad \text{for all } 1 \le i \le n.$ Since $|v_i|<2^j\leq p,$ we conclude from here
$v_i = 0 \quad \text{for all } 1 \le i \le n. $\\
From these it follows
\[
\sum_{i=1}^n x_i u_i \equiv 0 \pmod{p}
\quad \text{for all } (x_1,\ldots,x_n)\in \mathbb{Z}^n .
\] Hence,
\[
u_i = 0 \quad \text{for all } 1 \le i \le n .
\] Therefore, the vector $u$ is, in fact, a zero vector, which is absurd. Henceforth, the vector $u\in \lambda_1^*(z)D^*\cap \Lambda^*_z$ uniquely corresponds to an element $z\in Z'_j$. \\
The vector $(u_1,\ldots,u_n,v_1,\ldots, v_n)\in (2^jH_1/p)D^*\cap \Lambda^*_{z}$ obeys the inequality $\sum_{i=1}^n(|u_i|+|v_i|)H_i\leq 2^jH_1;$ hence, $|u_i|,|v_i|\leq 2^jH_1H_i^{-1}.$ So this implies 
\begin{equation}\label{Z'}
    |Z'_j|\leq \prod_{i=1}^n \max\{1,2^jH_1H_i^{-1}\}.
\end{equation}
Let us fix 
\[
j_i =
\begin{cases}
\log\!\left(\dfrac{H_{i+1}}{H_1}\right),
& \text{if } 1 \le i < n, \\[6pt]
\log\!\left(\dfrac{p}{H_1}\right)+1,
& \text{if } i = n.
\end{cases}
\]
So, we obtain from \eqref{Z'} that 
\begin{equation}\label{Z_j'}
|Z_j'| \le
\begin{cases}
2^{2j}, 
& \text{if } 1 \le j < j_1, \\[4pt]
2^{4j} H_1^{2} H_2^{-2}, 
& \text{if } j_1 \le j < j_2, \\[4pt]
\vdots \\[4pt]
2^{2k j} H_1^{2(k-1)} H_2^{-2} \cdots H_k^{-2}, 
& \text{if } j_{k-1} \le j < j_k, \\[4pt]
\vdots \\[4pt]
2^{2n j} H_1^{2n} |B|^{-2}, 
& \text{if } j_{n-1} \le j \le j_n .
\end{cases}
\end{equation}
For $n<s<2n$ we define $Z_j^s:=Z^s\cap Z'_j$. Now from the equation \eqref{f_0-boundnew2} we deduce
$$\sum_{z\in Z^s}f_0^2(z)\ll |B|^4p^{-2n}\sum_{j}\sum_{z\in Z_j^s}(\lambda_1^*(z))^{-2(2n-s)}.$$
As $\lambda_1^*(z)\asymp(2^jH_1/p)$ for $z\in Z'_j$, we obtain from above the following
\begin{equation*}
\begin{split}
    \sum_{z \in Z^s} f_0(z)&\ll |B|^4p^{-2n}\sum_{j}\sum_{z\in Z_j^s}(2^jH_1/p)^{-2(2n-s)}\\
           & \ll |B|^4p^{-2n}\sum_{k=1}^{n}\sum_{j_{k-1}\leq j<j_k}\sum_{z\in Z_j^s}(2^jH_1/p)^{-2(2n-s)}.
\end{split}
\end{equation*}
Now, from the equation \eqref{Z_j'} we obtain:

\begin{align*}
\sum_{Z^s} f_0 
\ll\; & |B|^4 p^{-2n}
\sum_{k=1}^{n}
\sum_{j_{k-1} \le j < j_k}
2^{2jk}
H_1^{2(k-1)} H_2^{-2} \cdots H_k^{-2}
\left( \frac{2^j H_1}{p} \right)^{-2(2n-s)} \\
\ll\; & |B|^2 p^{2(n-s)} H_1^{2s-4n}
\sum_{k=1}^{n}
H_1^{2k} H_{k+1}^2 \cdots H_n^2
\sum_{j_{k-1} \le j < j_k}
2^{2j(k+s-2n)}.
\end{align*}
Write
\[
M_k := \sum_{j_{k-1} \le j < j_k} 2^{2j(k+s-2n)},
\]
and divide the above sum according to $k < 2n-s$ and $k \geq 2n-s$, obtaining 
\begin{equation}\label{f_o}
\begin{split}
   &\sum_{z\in Z^s} f_0(z)\\
\ll \,& \underbrace{|B|^2 p^{2(n-s)} H_1^{2s-4n}\sum_{k < 2n-s}
H_1^{2k} H_{k+1}^2 \cdots H_n^2M_k}_{S_1}+\underbrace{|B|^2 p^{2(n-s)} H_1^{2s-4n}\sum_{k \geq 2n-s}
H_1^{2k} H_{k+1}^2 \cdots H_n^2M_k}_{S_2}.
\end{split}
\end{equation}
If $k < 2n-s$, then $k+s-2n < 0$, and hence
\[
M_k \ll \sum_{j_{k-1} \le j < j_k} 2^{2j(k+s-2n)} \ll \left(\frac{H_k}{H_1}\right)^{2(k+s-2n)}\log p.
\]
Therefore,
\begin{equation}\label{s1_acta}
    \begin{split}
        S_1&\ll |B|^2 p^{-2(s-n)}\sum_{k=1}^{2n-s-1}
H_k^{-2(2n-k-s)} H_{k+1}^2 \cdots H_n^2\;\log {p}  \\
          &\ll \; |B|^2\log{p}.\\
    \end{split}
\end{equation}
The last inequality comes from the following assumptions: $H\leq H_1\leq H_2\leq \ldots\leq H_{n-2}\leq H_{n-1}\leq H_n<\sqrt{p/2},\;H_{n-2}<2H$ and $s>n$.\\
If $k \geq 2n-s$, then $k+s-2n \geq 0$, and
\[
M_k \ll 2^{2j_k(k+s-2n)} \ll \left(\frac{H_{k+1}}{H_1}\right)^{2(k+s-2n)}.
\]
Substituting this bound, we obtain
\begin{align*}
S_2 
&\ll |B|^2 p^{2(n-s)} H_1^{2s-4n}
\sum_{k \ge 2n-s}
H_1^{2k} H_{k+1}^2 \cdots H_n^2
\left(\frac{H_{k+1}}{H_1}\right)^{2(k+s-2n)} \\
&\ll |B|^2
\sum_{k \ge 2n-s}
\left(\frac{H_{k+1}}{p}\right)^{2(k+s-2n)}
\left(\frac{H_{k+1}}{p}\right)^2 \cdots
\left(\frac{H_n}{p}\right)^2.
\end{align*}
Since $H_1 \le \cdots \le H_{n} \ll p$, it follows that
\begin{equation}\label{s2_acta}
S_2 \ll |B|^2.
\end{equation}
Combining \eqref{s1_acta} and \eqref{s2_acta} in \eqref{f_o}, we conclude that
\[
\sum_{z \in Z^s} f_0^2(z) \ll |B|^2 \log p,
\]
which establishes the desired bound for $n < s < 2n$ and completes the proof of Lemma~\ref{l1}.
\subsection{Proof of Lemma \ref{l2}}
The proof follows the argument of \cite[Lemma 2]{GB}, with the necessary modifications for general $n$. We sketch the main steps.

Fix $z \in \mathbb{F}_p^*$ and write
\[
x=\sum_{i=1}^n x_i \omega_i,\qquad y=\sum_{i=1}^n y_i \omega_i.
\]
Then $xz=y$ is equivalent to $zx_i \equiv y_i \pmod p$ for $1\le i\le n$. Hence,
\[
f_0(z)=\prod_{i=1}^n f_i(z), \qquad 
f_i(z)=|\{(x_i,y_i)\in[-H_i,H_i]^2: x_i z \equiv y_i \pmod p\}|.
\]
It follows that
\begin{equation}\label{L2short}
L_2=\sum_{z\in\mathbb{F}_p^*} f_0^2(z)
\le \prod_{i=1}^n \left(\sum_{z\in\mathbb{F}_p^*} f_i^2(z)\right).
\end{equation}
The sum $\sum_{z\in \mathbb{F}_p^*} f_i^2(z)$ appears in the proof of \cite[Lemma 2]{GB}; for completeness, we briefly indicate a few steps.
To estimate $\sum_{z\in \mathbb{F}_p^*} f_i^2(z)$, we associate to each $z$ the lattice
\[
\Lambda_z=\{(x,y)\in\mathbb{Z}^2:\ y\equiv xz \!\!\pmod p\},
\]
of covolume $\mu(\mathbb{R}^2/\Lambda_z)$. Let $\lambda_1(z),\lambda_2(z)$ denote its successive minima for all $z\in \mathbb{F}^*_p$. Then by Proposition \ref{Mkow} and Lemma \ref{LCE}
we get:
\[
\lambda_1(z)\lambda_2(z)\gg \frac{p}{H_i^2},
\qquad 
f_i(z)\ll \prod_{j=1}^2 \max\{1,\lambda_j^{-1}(z)\}.
\]
We partition according to the size of $\lambda_1(z)$ and estimate by counting lattice points in suitable dilates of $[-H_i,H_i]^2$ (see \cite[Lemma 2]{GB}). This yields
\[
\sum_{z\in\mathbb{F}_p^*} f_i^2(z)\ll H_i^2 \log p,
\]
using $H_i\le \sqrt{p}$. Substituting the above into \eqref{L2short}, we obtain
\[
L_2 \ll \prod_{i=1}^n H_i^2 (\log p)^n
= |B|^2 (\log p)^n,
\]
which completes the proof.
\section{Proof of the main theorem}
We shall describe the proof of Theorem \ref{MT} here. Let $H\leq \sqrt{p/2}$ be a natural number and let $\chi$ be a multiplicative character of $\mathbb{F}_{p^n}$. Define
\[
D(H,\chi):=\sup_{B}\frac{\left|\sum_{x \in B}\chi(x)\right|}{|B|},
\]
where the supremum is taken over all boxes $B$ of the form in \eqref{B} whose side lengths satisfy
\begin{equation}\label{H}
    H \leq H_i \leq 2H, \qquad i=1,2,\ldots,n-2.
\end{equation}
We first observe that if the side lengths of $B$ are large, namely
\[
H \leq H_i \leq p, \qquad i=1,2,\ldots,n-2,
\]
then $B$ can be partitioned into a disjoint union of boxes $B_{\alpha}$ whose side lengths satisfy \eqref{H}. By the assumption $H_i\geq (p^{n(1/4+\varepsilon)}/H_{n-1}H_{n})^{1/(n-2)}$ for all $1\leq i\leq n$, we can take $H=\left[(p^{n(1/4+\varepsilon)}/H_{n-1}H_{n})^{1/(n-2)}\right]$. Consequently, by the definition of $D(H,\chi)$,
\[
\left|\sum_{x \in B}\chi(x )\right| \leq D(H,\chi)\,|B|.
\]
Therefore, without loss of generality, it suffices to establish the following bound
\[
\left|\sum_{x \in B}\chi(x )\right| \leq \,|B|\, 
p^{-\varepsilon^2
\frac{1-\frac{1}{2n}}{\left(1+\frac{1}{4n}\right)\left(2-\frac{1}{2n}\right)}},
\]
where the side lengths of the box $B$ satisfy \eqref{H} with $H=\left[(p^{n(1/4+\varepsilon)}/H_{n-1}H_{n})^{1/(n-2)}\right]$.

\medskip

\emph{Case 1. $H_n<\sqrt{p/2}$.}

\medskip
\noindent

By our previous assumption \eqref{H} with the above choice of $H$, $|B| \asymp p^{n(1/4+\varepsilon)}$. Let $\delta=\delta(\varepsilon)>0$ be a number which will be chosen later. Set
\[I=[1,p^{\delta}] \cap \mathbb{Z}\text{  and  } B_0=\left\{\sum_{i=1}^{n}x_i\omega_i:x_i\in [0,p^{-2\delta}H_i]\cap \mathbb{Z},1\leq i\leq n\right\}.\]
We follow the amplification of $B$ by $B_0I\subseteq \left\{\sum_{i=1}^{n-1}x_i\omega_i:x_i\in [0,p^{-\delta}H_i]\cap \mathbb{Z},1\leq i\leq n\right\}$ as in \cite[Equation (3.1), (3.2)]{GB} and obtain
\begin{equation}\label{FI}
\sum_{x \in B}\chi(x )=\frac{1}{|B_0||I|}\sum_{x \in B, y\in B_0, z\in I}\chi(x+yz) +O(np^{-\delta}|B|).
\end{equation}
Also,
\begin{equation}\label{SI}
\begin{split}
\left|\sum_{x \in B, y\in B_0, z\in I}\chi(x+yz)\right|\leq &\sum_{x \in B, y\in B_0\setminus\{0\}}\left|\sum_{z\in I}\chi(xy^{-1}+z)\right|+ |B||I|\\
= & \sum_{u\in \mathbb{F}_{p^n}}\omega(u)\left|\sum_{z\in I}\chi(u+z)\right|+ |B||I|\\
\end{split}
\end{equation}
where, $$\omega(u)=\left|\{(x,y)\in B\times (B_0\setminus \{0\}): xy^{-1}=u\}\right|.$$
Substituting the equation \eqref{SI} in the equation \eqref{FI} and using H\"older's inequality in the $u$ sum above we obtain 
\begin{equation}\label{TI}
\begin{split}
\sum_{x \in B}\chi(x)
\leq &\frac{1}{|B_0||I|}\left(\sum_{u\in \mathbb{F}_{p^n}}\omega(u)\right)^{1-1/r} \left(\sum_{u \in \mathbb{F}_{p^n}}\omega^2(u)\right)^{1/(2r)}\left(\sum_{u\in \mathbb{F}_{p^n}}\left|\sum_{z\in I}\chi(u+z)\right|^{2r}\right)^{1/(2r)}\\
+& O(np^{-\delta}|B|) +|B||B_0|^{-1}.\\
\end{split}
\end{equation}

By the trivial estimate, it follows that
\begin{equation}\label{A}
\sum_{u\in \mathbb{F}_{p^n}}\omega(u)\leq |B||B_0|.
\end{equation}
On the other hand, $\omega(0)\leq |B_0|$ and so $\omega(0)^2\leq |B||B_0|$ as $|B_0|\leq B$.
Using Cauchy-Schwarz inequality and Lemma \ref{KL}, it follows that
\begin{equation*}
\begin{split}
\sum_{u\in \mathbb{F}^*_{p^n}}\omega^2(u)&=|\{(x_1,x_2,y_1,y_2)\in B\times B\times B_0\times B_0:x_1y_2=x_2y_1\neq 0\}|\\
&\leq \sum_{\nu\in \mathbb{F}^*_{p^n}}|\{(x_1,x_2)\in B^2:\frac{x_1}{x_2}=\nu\}||\{(y_1,y_2)\in B_0^2:\frac{y_1}{y_2}=\nu\}|\\
&\leq E(B)^{1/2}E(B_0)^{1/2}\ll |B||B_0|(\log p)^n.\\
\end{split}
\end{equation*}
Combing the above two, we deduce 
\begin{equation}\label{FvI}
\sum_{u\in \mathbb{F}_{p^n}}\omega^2(u)\ll|B||B_0|(\log p)^n.
\end{equation}
Furthermore, from Lemma \ref{WE2} we can conclude 
\begin{equation}\label{WL}
\left(\sum_{u\in \mathbb{F}_{p^n}}\left|\sum_{z\in I}\chi(u+z)\right|^{2r}\right)^{1/2r}\ll p^{n/(4r)}|I|+ p^{n/(2r)}|I|^{1/2}.
\end{equation}
Substituting the estimates \eqref{A}, \eqref{FvI}, and \eqref{WL} into \eqref{TI} and finally combining this with \eqref{FI}, we obtain
\begin{equation}
\begin{split}
 \sum_{x\in B} \chi(x) &\ll \frac{1}{|B_0||I|}(|B||B_0|)^{1-\frac{1}{r}}(|B||B_0|(\log p)^n)^{1/2r}\left(p^{n/4r}|I|+p^{n/2r}|I|^{1/2}\right) +|B||B_0|^{-1}+ n|B| p^{-\delta}\\
 &= |B|(|B||B_0|)^{-1/2r}(\log p)^{n/2r}\left(p^{n/4r}+ p^{n/2r}|I|^{-1/2}\right)+p^{2n\delta}+ np^{-\delta}|B|\\
\end{split}    
\end{equation}
Now we fix $\delta=n/2r$ with $r$ large enough and substitute the size of $I$ with $p^\delta$ to obtain
\begin{equation*}
\sum_{x\in B} \chi(x)\ll |B|p^{-2\delta(\varepsilon -\delta)}(\log p)^{\delta}+p^{2n\delta}+ n|B|p^{-\delta}    
\end{equation*}
In fact, we choose $r$ so that $\delta=n/2r$ is close to $\varepsilon/2$. More precisely, let $r$ be the closest integer to the number $n\varepsilon^{-1}$. So, $r$ can be written as $r=n\varepsilon^{-1}+\frac{1}{2}\theta$ where $|\theta|\leq 1$. Substituting the value of $r$ in $\delta$ and by the assumption $|B|\gg p^{n/4+n\varepsilon}$ we deduce that
\begin{equation*}
    \left|\sum_{x \in B}\chi(x )\right|\ll_{\varepsilon}|B|p^{-2\delta(\varepsilon-\delta)}(\log p)^{\delta}+n|B|p^{-\varepsilon/3}.
\end{equation*}
Finally,
\begin{equation}
    \begin{split}
        2\delta(\varepsilon-\delta)\geq 2.\frac{\varepsilon}{2+\frac{1}{2n}}\left(\varepsilon-\frac{\varepsilon}{2-\frac{1}{2n}}\right)\geq \frac{\varepsilon^2\left(1-\frac{1}{2n}\right)}{\left(1+\frac{1}{4n}\right)\left(2-\frac{1}{2n}\right)}.
    \end{split}
\end{equation}
Therefore, it is evident that 
\begin{equation*}
    \left|\sum_{x \in B}\chi(x)\right|\ll_{\varepsilon,n}|B|p^{-\varepsilon^2\frac{8n^2-4n}{16 n^2-1}}
\end{equation*}

\emph{Case 2. $\sqrt{p/2}<H_n<p^{1/2+\varepsilon/2}$}

\medskip
\noindent

Here we divide each edge of the box into $O(p^{\varepsilon/2})$ equal pieces of length less than $\sqrt{p/2}$. So, $B$ can be divided into $O((p^{\varepsilon/2})^n)$ smaller boxes $B_{\alpha}$ of volume $\gg (p^{-\varepsilon/2})^np^{n(1/4+\varepsilon)}=p^{n(1/4+\varepsilon/2)}.$ Also, the boxes $B_{\alpha}$ have side lengths $H'_1,H'_2,\ldots, H'_n$, respectively, where $H_n'=H_np^{-\varepsilon/2}\geq H_{n-1}'=H_{n-1}p^{-\varepsilon/2}\geq \cdots\geq H_1'=H_1p^{-\varepsilon/2}\geq \left(\frac{p^{(1/4+\varepsilon/2)n}}{H'_{n-1}H'_n}\right)^{1/(n-2)}$ and $H_n'\leq \sqrt{p/2}$. Now, the character sum estimate over $B_{\alpha}$ satisfies a non-trivial upper bound by our previous case for all $\alpha$
\[\left|\sum_{x \in B_{\alpha}}\chi(x )\right|\ll_{\varepsilon}|B_{\alpha}|p^{-(\varepsilon^2/4)\frac{8n^2-4n}{16 n^2-1}}.\]
Finally, adding the contributions of each $B_{\alpha}$ it follows
\[\left|\sum_{x \in B}\chi(x )\right|\ll_{\varepsilon}|B|p^{-\varepsilon^2\frac{2n^2-
n}{16 n^2-1}}.\]

\medskip
\noindent

\medskip

Next, we shall prove Theorem \ref{MT}(ii). The same argument as in Case 1 and Case 2 also applies here. The only remaining case is $H_n>p^{1/2+\varepsilon/2}$ which will be discussed below in Case 3.

\medskip
\noindent

\emph{Case 3. (The {\rm (ii)} Part)} $H_n>p^{1/2+\varepsilon/2}.$

\medskip

We invoke the following estimate due to Katz \cite{Katz}.
\begin{Proposition}[Katz]\label{KT}
Let $\chi$ be a non-trivial multiplicative character of $\mathbb{F}_{p^n}$ and let
$g\in\mathbb{F}_{p^n}$ generate the extension
$\mathbb{F}_{p^n}=\mathbb{F}_p(g)$. Then, for any interval
$I\subseteq [1,p]\cap\mathbb{Z}$,
\[
\left|\sum_{t\in I}\chi(g+t)\right|
\le c(n)\sqrt{p}\log p .
\]
\end{Proposition}

\begin{proof}
See \cite[Theorem~1]{Katz}.
\end{proof}

\noindent\medskip
We write
\begin{equation}\label{SB}
\sum_{x \in B}\chi(x )
=
\sum_{(x_1,\ldots,x_{n-1})\in I_1\times\cdots\times I_{n-1}}
\sum_{x_n\in I_n}
\chi\!\left(
x_1\frac{\omega_1}{\omega_n}
+\cdots+
x_{n-1}\frac{\omega_{n-1}}{\omega_n}
+x_n
\right),
\end{equation}
where $I_i=[N_i+1,N_i+H_i]\cap\mathbb{Z}$.\\
Define
\[
\Omega :=
\Bigl\{
(x_1,\ldots,x_{n-1})\in I_1\times\cdots\times I_{n-1} :
\mathbb{F}_p\!\left(
x_1\frac{\omega_1}{\omega_n}
+\cdots+
x_{n-1}\frac{\omega_{n-1}}{\omega_n}
\right)
\neq \mathbb{F}_{p^n}
\Bigr\}.
\]
Let us consider the contributions of $(x_1,\ldots,x_{n-1})\in \Omega$ and $(x_1,\ldots,x_{n-1})\notin \Omega$ separately.

\noindent\medskip
\textbf{Contribution 1: $(x_1,\ldots, x_{n-1})\notin \Omega $.}

\noindent\medskip
Now, for every $(x_1,\ldots,x_{n-1})\notin \Omega$, the element
\[
x_1\frac{\omega_1}{\omega_n}
+\cdots+
x_{n-1}\frac{\omega_{n-1}}{\omega_n}
\]
generates $\mathbb{F}_{p^n}$ over $\mathbb{F}_p$. Applying Proposition~\ref{KT} and using the hypothesis $H_n>p^{1/2+\varepsilon/2}$, we obtain
\[
\left|
\sum_{x_n\in I_n}
\chi\!\left(
x_1\frac{\omega_1}{\omega_n}
+\cdots+
x_{n-1}\frac{\omega_{n-1}}{\omega_n}
+x_n
\right)
\right|
\ll
\sqrt{p}\log p
\le
H_n\,p^{-\varepsilon/2}\log p .
\]
Summing over $x_1,\ldots, x_{n-1}$ trivially it yields
\[
\left|\sum_{x \in B}\chi(x )\right|
\ll_{\varepsilon}
|B|\,p^{-\varepsilon/3}.
\]

\medskip

\noindent\textbf{Contribution 2: $(x_1,\ldots,x_{n-1}) \in \Omega$.}

\noindent\medskip
We split it into two sub-cases.

\medskip
\textbf{Case A: $(n,6)=1$.}
In this case, the sum in \eqref{SB} is bounded by
\begin{equation}\label{SM}
\widetilde{\sum}=\sum_{(x_1,\ldots,x_{n-1})\in \Omega}
\sum_{x_n\in I_n}
\chi\!\left(
x_1\frac{\omega_1}{\omega_n}
+\cdots+
x_{n-1}\frac{\omega_{n-1}}{\omega_n}
+x_n
\right).
\end{equation}
A trivial estimate gives 
\begin{equation}\label{triv}\widetilde{\sum} \;\ll\; p\,|\Omega|.
\end{equation}
Hence it suffices to estimate $|\Omega|$ which will be carried out.\\
Let $G$ run over all non-trivial subfields of $\mathbb{F}_{p^n}$. Then
\begin{equation}\label{G}
|\Omega|
\;\leq\;
\sum_{G\subsetneq \mathbb{F}_{p^n}}
\left|
G \cap \mathrm{Span}_{\mathbb{F}_p}\!\left(
\frac{\omega_1}{\omega_n},\ldots,\frac{\omega_{n-1}}{\omega_n}
\right)
\right|.
\end{equation}
Since $(n,6)=1$, any proper subfield $G \subsetneq \mathbb{F}_{p^n}$ satisfies
$
[\mathbb{F}_{p^n}:G]\geq 5,$ and hence $
[G:\mathbb{F}_p]\leq \frac{n}{5}.$
Moreover, as
\[
1 \notin \mathrm{Span}_{\mathbb{F}_p}\!\left(
\frac{\omega_1}{\omega_n},\ldots,\frac{\omega_{n-1}}{\omega_n}
\right),
\]
we obtain
\[
\dim_{\mathbb{F}_p}
\left(
G \cap \mathrm{Span}_{\mathbb{F}_p}\!\left(
\frac{\omega_1}{\omega_n},\ldots,\frac{\omega_{n-1}}{\omega_n}
\right)
\right)
\leq \frac{n}{5}-1.
\]
It follows from the above that $|\Omega| \;\ll\; p^{\frac{n}{5}-1},$
and consequently by \eqref{triv}, $\widetilde{\sum} \;\ll\; p^{\frac{n}{5}}.$
Finally, recalling that $|B|\gg p^{\frac{n}{4}}$, we deduce $\widetilde{\sum} \;\ll\; p^{-\frac{n}{20}}\,|B|.$
This completes the proof in this case.

\medskip

\textbf{Case B:\,($(n,6)\neq 1 \text{ with } \chi_{\restriction \mathbb{F}_{p^{n/r}}} \neq \chi_0$ for r=2 and 3)}

\medskip
As in Case~A, if 
\[
\left[\mathbb{F}_{p^n}:\mathbb{F}_p\!\left(\sum_{j=1}^{n-1} x_j \frac{\omega_j}{\omega_n}
\right)\right]\geq 5,
\]
then the character sum can be estimated using \eqref{SM} and \eqref{G}, yielding the same bound as before. It therefore remains to treat the cases corresponding to subfields of index $r=2$ and $r=3$.
So, it is enough to estimate 
\begin{equation}\label{S}
\sum' :=
\sum_{(x_1,\ldots,x_{n-1})\in \Omega_r}
\sum_{x_n\in I_n}
\chi\!\left(
x_1\frac{\omega_1}{\omega_n}
+\cdots+
x_{n-1}\frac{\omega_{n-1}}{\omega_n}
+x_n
\right),
\end{equation}
for $r=2 \text{ and } 3$ where
\[
\Omega_r :=
\Bigl\{
(x_1,\ldots,x_{n-1})\in I_1\times\cdots\times I_{n-1}
:
\mathbb{F}_p\!\left(
\sum_{j=1}^{n-1} x_j \frac{\omega_j}{\omega_n}
\right)
\subseteq \mathbb{F}_{p^{n/r}}
\Bigr\}.
\]
Since $1,\omega_1/\omega_n,\ldots,\omega_{n-1}/\omega_n$ are linearly independent over $\mathbb{F}_p$, at most $(n/r)-1$ of the ratios $\omega_j/\omega_n$ lie in $\mathbb{F}_{p^{n/r}}$. Reordering if necessary, we assume
\[
\omega_j/\omega_n \in \mathbb{F}_{p^{n/r}} \ (1\le j\le k),
\qquad
\omega_j/\omega_n \notin \mathbb{F}_{p^{n/r}} \ (k<j<n),
\]
with $k\le n/r$. Fixing $x_1,\ldots,x_{n-2}$, there is exactly one choice of
$x_{n-1}$ for which
\[
\sum_{j=1}^{n-1} x_j \frac{\omega_j}{\omega_n}
\in \mathbb{F}_{p^{n/r}},
\]
since otherwise subtraction yields a contradiction to
$\omega_{n-1}/\omega_n \notin \mathbb{F}_{p^{n/r}}$.
Hence
\[
|\Omega_r|\le |I_1|\cdots|I_{n-2}|,
\]
and therefore
\[
\sum' \le \frac{|B|}{H_{n-1}}.
\]
If $H_{n-1}>p^{\varepsilon/2}$ the claim follows. Else, $H_{n-1}\leq p^{\varepsilon/2}$. Therefore, 
\begin{equation}\label{nwe}
H_{k+1}\cdots H_{n-1}\le p^{\varepsilon n/2}.
\end{equation}
Define
\[
W_r :=
\left\{
\sum_{i=1}^{k} x_i \frac{\omega_i}{\omega_n} + x_n
:
x_i\in I_i \ (1\le i\le k), \ x_n\in I_n
\right\}
\subseteq \mathbb{F}_{p^{n/r}}.
\]
From \eqref{nwe} we obtain
\begin{equation}\label{Wsize}
|W_r|
=
\frac{|B|}{H_{k+1}\cdots H_{n-1}}
\ge
p^{(1/4+\varepsilon)n-\varepsilon n/2}
=
p^{(1/4+\varepsilon/2)n}.
\end{equation}
From the assumption of Theorem \ref{MT}(ii), we have that $\chi$ restricted to $\mathbb{F}_{p^{n/r}}$ is non-principal. So, by Proposition \ref{PV}, one deduces that
\[
\sum_{y\in W_r}\chi(y+z)
\leq
(\log p)^{n/r}
|\mathbb{F}_{p^{n/r}}|^{1/2}
\leq
(\log p)^{n/2} p^{n/4}\leq p^{n\varepsilon/4}p^{n/4}.
\]
Combining this with \eqref{Wsize}, we deduce
\[
\sum_{y\in W_r}\chi(y+z)
\le
p^{-\varepsilon n/4}|W_r|.
\]
Consequently, by trivially estimating other intervals in \eqref{S}, it follows 
\[
\sum'
\le
H_{k+1}\cdots H_{n-1}\,
p^{-\varepsilon n/4}\,
|W_r|
=
p^{-\varepsilon n/4}|B|,
\]
which completes the argument for $\chi\restriction_{p^{n/r}}\neq\chi_0$. This ends the proof of Theorem \ref{MT}.
\begin{remark}
Suppose that the hypotheses of Theorem~\ref{MT}(ii) concerning the non-triviality of the restrictions of $\chi$ are not satisfied for $\gcd(6,n)\neq 1$. Then the argument used in the proof of Theorem~\ref{MT}(ii) yields the following estimates.

\medskip
\noindent\textbf{Case 1.}
If $6\mid n$ and $\chi \restriction_{p^{n/2}}=\chi_0,\,\chi \restriction_{p^{n/3}}=\chi_0$. then we consider 
\begin{equation*}
\begin{split}
&\left|\sum_{x\in B} \chi(x)\right|\\=& \left|(\mathbb{F}_{p^{n/2}}\cup\mathbb{F}_{p^{n/3}})\cap \frac{1}{\omega_n}B\right|+\sum_{(x_1,\ldots,x_{n-1})\in I_1\times\cdots\times I_{n-1}\setminus \Omega_2\cup\Omega_3}
\sum_{x_n\in I_n}
\chi\!\left(
x_1\frac{\omega_1}{\omega_n}
+\cdots+
x_{n-1}\frac{\omega_{n-1}}{\omega_n}
+x_n
\right)\\
=&\left|(\mathbb{F}_{p^{n/2}}\cup\mathbb{F}_{p^{n/3}})\cap \frac{1}{\omega_n}B\right| +O_n(|B|p^{-\delta})\\
\end{split}
\end{equation*}
Applying the inclusion-exclusion principle
\[
\left|\sum_{x\in B} \chi(x)\right|=|\omega_n \mathbb{F}_{p^{n/2}}\cap B|+|\omega_n \mathbb{F}_{p^{n/3}}\cap B|-|\omega_n \mathbb{F}_{p^{n/6}}\cap B|+O_n(p^{-\delta}|B|).
\]

\medskip
\noindent\textbf{Case 2.}
If $2\mid n,\, \gcd(3,n)=1$ and $\chi \restriction_{p^{n/2}}=\chi_0$ or $6\mid n,\,\chi \restriction_{p^{n/3}}\neq\chi_0$ and $\chi \restriction_{p^{n/2}}=\chi_0$,  then by a similar argument
\[
\left|\sum_{x\in B} \chi(x)\right|=|\omega_n \mathbb{F}_{p^{n/2}}\cap B|+O_n(p^{-\delta}|B|).
\]

\medskip
\noindent\textbf{Case 3.}
If $3\mid n,\, \gcd(2,n)=1$ and $\chi \restriction_{p^{n/3}}=\chi_0$ or  $6\mid n,\,\chi \restriction_{p^{n/3}}=\chi_0$ and $\chi \restriction_{p^{n/2}}\neq\chi_0$,  the same argument yields
\[
\left|\sum_{x\in B} \chi(x)\right|=|\omega_n \mathbb{F}_{p^{n/3}}\cap B|+O_n(p^{-\delta}|B|).
\]
\end{remark}
{\bf Acknowledgments.} The author expresses his sincere gratitude to Professor Stephan Baier for the valuable conversations. The author thanks the Ramakrishna Mission Vivekananda Educational and Research Institute for an excellent work environment. The research of the author was supported by a CSIR Ph.D. fellowship under file number 09/0934(13170)/2022-EMR-I.


\begin{thebibliography}{20}	

\bibitem{Ban} W. Banaszczyk, {\it Inequalities for convex bodies and polar reciprocal lattices in $R^n$}, Discrete Comput. Geom. 13 (2), 217-231 (1995).

\bibitem{BHW} U. Betke, M. Henk, J. M. Wills, {\it Successive-minima-types inequalities}, Discrete Comput. Geom., 9:2 (1993), 165-175 

\bibitem{B1} D. A. Burgess, {\it On character sums and primitive roots}, 
Proc. London Math. Soc. (3) 12 (1962), 179-192.

\bibitem{B2} D. A. Burgess, {\it Character sums and primitive roots in finite fields}, 
Proc. London Math. Soc. (3) 37 (1967), 11-35.

\bibitem{C1} M. C. Chang, {\it On a question of Davenport and Lewis and new character sum bounds in finite fields}, 
Duke Math. J. 145:3 (2008), 409-422.

\bibitem{Cha} M.-C. Chang, {\it Burgess inequality in $\mathbb{F}_{p^2}$}, 
Geom. Funct. Anal. 19 (2009), 1001--1016.

\bibitem{Ch} A. Chattopadhyay, {\it Character sums over co-dimension one sub-lattices in finite field extensions}, https://doi.org/10.48550/arXiv.2505.19654.

\bibitem{DL} H. Davenport, D. Lewis, {\it Character sums and primitive roots in finite fields}, 
Rend. Circ. Matem. Palermo-Serie II-Tomo XII-Anno (1963).


\bibitem{JH} J. Friedlander, H. Iwaniec, {\it Estimates of character sums}, 
Proc. Amer. Math. Soc. 119:2 (1993), 265-372.

\bibitem{GB} M. Gabdullin, {\it Estimates for character sums in finite fields of order $p^2$ and $p^3$}, Proc. Steklov Inst. Math., vol. 303 (2018), 36-49.

\bibitem{K} A. A. Karatsuba, {\it Estimates of character sums}, 
Math. USSR Izv. 4 (1970), no.~1, 19--29.

\bibitem{Katz} N. Katz, {\it An estimate of character sums}, Journal of the American Mathematical Society Vol. 2, No 2 1989, 197-200.

\bibitem{Kon} S. V. Konyagin, {\it Estimates for character sums in finite fields (Russian)}, 
Mat. Zametki 88 (2010), no. 4, 529-542; translation in Math. Notes 88 (2010), no. 3-4, 503-515.

\bibitem{Sch} W.M. Schmidt {\it Equations over Finite Fields: An Elementary Approach},
Lecture Notes in Mathematics. 536. Berlin-Heidelberg-New York: Springer-Verlag. ix, 267 p. (1976).

\bibitem{TV} T. Tao, V. Vu, {\it Additive Combinatorics}, Cambridge University Press, 2006.

\end{thebibliography}
\end{document}